\documentclass[12pt]{amsart}
\usepackage{amsfonts, amssymb, mathrsfs, amsmath}
\usepackage{float, fullpage, verbatim, bbm}
\usepackage[colorlinks=true]{hyperref}
\usepackage{breakurl}
\usepackage{multicol}
\usepackage{tikz}
\usetikzlibrary{calc}
\usepackage[english]{babel}
\usepackage{etoolbox}
\makeatletter
\let\ams@starttoc\@starttoc
\makeatother
\usepackage[parfill]{parskip}
\makeatletter
\let\@starttoc\ams@starttoc
\patchcmd{\@starttoc}{\makeatletter}{\makeatletter\parskip\z@}{}{}
\makeatother
\usetikzlibrary{arrows}
\usepackage{multirow,tabularx}
\usepackage{xfrac}
\usepackage{longtable}
\usepackage[colorlinks=true]{hyperref}
\usepackage{booktabs}
\usepackage{mathtools}
\usepackage{subfigure}
\usepackage{geometry}
\usepackage{setspace}
\usepackage[english]{babel}
\usepackage[T1]{fontenc}
\usepackage[utf8]{inputenc}
\usepackage{enumitem}
\usepackage{graphicx}
\usepackage{fancyhdr}
\usepackage{amsthm}
\usepackage{geometry,rotating,tabstackengine}
\usepackage[arrow, matrix, curve]{xy}
\usepackage{tikz}
\usepackage{pgfplots}
\usepackage{amsfonts}
\usepackage{bbm}
\usepackage{algorithm}
\usepackage{algorithmic}           
\usepackage{listings}
\usetikzlibrary{shapes,positioning}
\theoremstyle{definition}
\newtheorem{theorem}{Theorem}[section]
\newtheorem{definition}[theorem]{Definition}
\newtheorem{lemma}[theorem]{Lemma}
\newtheorem{corollary}[theorem]{Corollary}
\newtheorem{proposition}[theorem]{Proposition}
\theoremstyle{remark}
\newtheorem{example}[theorem]{Example}
\newtheorem{remark}[theorem]{Remark}

\newtheorem{question}[theorem]{Question}

\DeclareMathOperator{\LMP}{LMP}
\DeclareMathOperator{\GMP}{GMP}
\DeclareMathOperator{\VGMP}{VGMP}

\renewcommand{\epsilon}{\varepsilon}
\newcommand{\rank}{\operatorname{rank}}
\newcommand{\codim}{\operatorname{codim}}
\newcommand{\Der}{\operatorname{Der}}

\newcommand\CA{{\mathscr A}}
\newcommand\CE{{\mathscr E}}

\newcommand\ZZ{{\mathbb Z}}
\newcommand\QQ{{\mathbb Q}}

\newcommand\KK{{\mathbb K}}
\newcommand\CB{{\mathscr B}}

\newcommand\BBQ{{\mathbb Q}}

\newcommand\BBZ{{\mathbb Z}}

\subjclass{52C35, 32S22, 51F15}

\begin{document}
\title[Extendability of the $B_2$-arrangement]
{Extendability of the $B_2$-arrangement}

\author[T.~Hoge]{Torsten Hoge}
\address
{Fakult\"at f\"ur Mathematik,
	Ruhr-Universit\"at Bochum,
	D-44780 Bochum, Germany}
\email{torsten.hoge@rub.de}

\author[S.~Maehara]{Shota Maehara}
\address
{Joint Graduate School of Mathematics for Innovation, Kyushu University, 744 Motooka Nishi-ku Fukuoka 819-0395, Japan}
\email{maehara.shota.027@s.kyushu-u.ac.jp}

\author[S.~Wiesner]{Sven Wiesner}
\address
{Fakult\"at f\"ur Mathematik,
	Ruhr-Universit\"at Bochum,
	D-44780 Bochum, Germany}
\email{sven.wiesner@rub.de}

\begin{abstract}
    Let $(\CA,m)$ be a free multiarrangement, and let $\CE$ be an extension of $(\CA,m)$. It is well known that if $\CA$ is the Coxeter arrangement of type $A_2$, then a free extension of $(\CA,m)$ always exists. In this work, we demonstrate that if $\CA$ is the Coxeter arrangement of type $B_2$, there exist infinitely many multiplicities for which no free extension of $(\CA,m)$ exists. This result has immediate consequences for the existence of free extensions in higher rank.
\end{abstract}

\maketitle

\section{Introduction}
For a given hyperplane arrangement $\CA$, Saito introduced the module of logarithmic vector fields $D(\CA)$ in \cite{MR0586450}. A famous and still open conjecture by Terao asserts that the freeness of $D(\CA)$ depends solely on the combinatorial structure of $\CA$. It is known that if $D(\CA)$ is free, then the degrees of a homogeneous basis of $D(\CA)$ are uniquely determined, although the generators themselves need not be. These degrees are in one-to-one correspondence with the roots of the characteristic polynomial $\chi(\CA;t)$, whose definition depends only on the combinatorial properties of $\CA$. Ziegler introduced the notion of a multiarrangement $(\CA, m)$ in \cite{MR1000610}, where $\CA$ is a hyperplane arrangement and $m: \CA \to \ZZ_{\geq 0}$ is a multiplicity function. He generalized the module of logarithmic vector fields from $D(\CA)$ to $D(\CA,m)$, and showed that if $D(\CA)$ is free, then there exists a canonical multiarrangement $(\CA^H, m^H)$ such that $D(\CA^H, m^H)$ is a free module. The multiplicity function $m^H$ is defined in a purely combinatorial manner. Moreover, the degrees of a homogeneous basis of $D(\CA)$ coincide with the degrees of a homogeneous basis of $D(\CA^H,m^H)$. This holds despite the fact that, for an arbitrary multiarrangement $(\CA,m)$, the degrees of a minimal homogeneous set of generators for $D(\CA,m)$ are not determined by the combinatorial data of $\CA$. Subsequently, Yoshinaga posed the question in \cite{MR3025868} of whether, for every free multiarrangement $(\CA, m)$, there exists a free simple arrangement $\CE$ such that $(\CE^H,m^H)=(\CA, m)$ for some $H\in\CE$. If such an arrangement $\CE$ exists, then we call $\CE$ a \emph{free extension} of $(\CA, m)$. It is known that if $\CA$ is the Coxeter arrangement $\CA(A_2)$ of type $A_2$, and $m$ is an arbitrary multiplicity on $\CA$, then a free extension of $(\CA, m)$ always exists (\cite[Thm.~2.5]{MR3025868} and \cite[Lem.~5.8]{2406.19866}). 
Building on this, one might expect that a free extension exists for every free multiarrangement $(\CA, m)$. Given an arbitrary free multiarrangement $(\CA,m)$ and a corresponding free extension $\CE$, this would provide a means to determine the degrees of a homogeneous set of generators of $D(\CA,m)$ by analyzing the combinatorial structure of $D(\CE)$. In general, determining these degrees for $D(\CA,m)$ is a highly nontrivial problem; even among Coxeter arrangements of rank two, a complete classification is currently only known for the Coxeter arrangement of type $A_2$. In this paper, we focus on the Coxeter arrangement $\CA(B_2)$ of type $B_2$ and demonstrate that there exist infinitely many multiplicities $m_k$ for which no free extension of the corresponding multiarrangement $(\CA(B_2), m_k)$ exists. The main result of this article is as follows.

\begin{theorem}\label{theorem:MainResult}
    Let $\CA=\CA(B_2)$ be the Coxeter arrangement of type $B_2$ with $Q(\CA)=xy(x-y)(x+y)$. For $k\in\mathbb{Z}_{\geq 0}$ define the Coxeter multiarrangement $$Q(\CA,m_k)=x^2y^k(x-y)^1(x+y)^k.$$ If $k\geq 4$, then there does not exist a free extension of $(\CA,m_k)$.
\end{theorem}

As noted in Remark~\ref{remark:permutation}, the values of the multiplicity $m_k$ may be permuted among certain hyperplanes of $\CA(B_2)$ without affecting the validity of the result. We prove Theorem \ref{theorem:MainResult} by employing results of Abe and Yoshinaga, along with a classification of the degrees of basis elements for certain multiplicities on $\CA(B_2)$. As a straightforward corollary of Theorem \ref{theorem:MainResult}, we further obtain infinitely many free multiplicities on the Coxeter arrangement of type $B_n$ (for $n \geq 3$) that admit no free extensions. 

\begin{theorem}[Corollary \ref{corollary: implications for higher Bn}]\label{theorem 2}
    Let $(\CA,m)=(\CA(B_n),m)$ be the Coxeter multiarrangement of type $B_n$. If there exists a $X\in L(\CA)$, with $\rank(X)=2$ and $(\CA_X,m_X)=(\CA(B_2),m_k)$ (where $k\geq 4$ and $(\CA(B_2),m_k)$ as in Theorem \ref{theorem:MainResult}), then no free extension of $(\CA,m)$ exists.
\end{theorem}

In our final result, we present an application of Theorem \ref{theorem 2}, which yields an infinite family of inductively free multiplicities $n_k$ on the Coxeter arrangement $\CA(B_3)$ of type $B_3$, such that there does not exist a free extension of any of the pairs $(\CA(B_3),n_k)$.

\begin{theorem}[Corollary \ref{corollary: B3 non extendable example}]
    Let $$Q(\CA(B_3),n_k)=x^2 y^k (x-y)^1 (x+y)^k z^e (x-z)^f (x+z)^g (y-z)^h (y+z)^i, (k\geq 4),$$ be the Coxeter multiarrangement of type $B_3$. Fix the supersolvable filtration $$\{\ker(x)\}\subset \{\ker(x), \ker(y), \ker(x-y), \ker(x+y)\}\subset \CA(B_3).$$ Then $(\CA(B_3),n_k)$ satisfies the conditions of Theorem \ref{theorem: free vertex} if and only if $e=f=g=h=i=1$. Moreover, each $(\CA(B_3),n_k)$ is inductively free with exponents $\exp(\CA(B_3),n_k)=(5,k+1,k+2)$. However, no free extension of $(\CA(B_3),n_k)$ exists.
\end{theorem}

\section{Preliminaries}
In this section we recall some definitions and elementary results. For this, we closely follow \cite{MR1217488}. 
\begin{definition}
Let $\KK$ be a field, and let $V$ be a vector space of dimension $\ell$ over $\KK$. A \emph{hyperplane} $H$ in $V$ is an affine subspace of dimension $\ell-1$. A \emph{hyperplane arrangement} (or simply an \emph{arrangement}) is a pair $\CA = (\CA, V)$, where $\CA$ is a finite set of hyperplanes in $V$.\\
We refer to $\CA$ as an $\ell$-arrangement to indicate that $\dim V=\ell$.
\end{definition}

\begin{definition}
1. We call a hyperplane arrangement $\CA$ \emph{central} if $\cap_{H \in \CA} H \neq \emptyset$. Therefore if $\CA$ is central, we will assume that each hyperplane contains the origin and is thus a subspace.\\
2. Let $V^*$ be the dual space of $V$ and $S = S(V^*)$ be the symmetric algebra of $V^*$. If $x_1, \ldots , x_\ell$ is a basis of $V^*$, then we identify $S$ with the polynomial algebra $S = \mathbb{K}[x_{1}, \ldots, x_\ell]$.
\end{definition}

All arrangements investigated in this paper are central. 

\begin{definition}
    Let $\CA$ be a hyperplane arrangement. Each hyperplane $H \in \CA$ is the kernel of a polynomial $\alpha_H$ of degree 1, defined up to a constant. A function $m:\CA\to \mathbb{Z}_{\geq0}$ is called a \emph{multiplicity} on $\CA$ and the pair $(\CA,m)$ is called a \emph{multiarrangement}. We define the \emph{cardinality} $\vert m\vert$ of $(\CA,m)$ as $\vert m\vert:=\sum_{H\in\CA}m(H)$. The product \[Q(\CA,m) = \prod_{H \in \CA} \alpha_H^{m(H)}\] is called the \emph{defining polynomial of $(\CA,m)$}. If $m\equiv 1$, we call $(\CA,m)$ a \emph{simple arrangement} and denote it by just $\CA$. If $\CA$ is central, then $\alpha_H$ is a linear form for all $H$ and $Q(\CA)$ is a homogeneous polynomial.
\end{definition}

We now introduce the \emph{lattice of intersections} $L(\CA)$ of an arrangement $\CA$, which will play a fundamental role in the proof of our main result.

\begin{definition}
	Let $L(\CA)$ be the set of all possible intersections of hyperplanes of $\CA$: \[L(\CA) = \big\{ \textstyle\bigcap_{H\in \CA'}H~\vert ~\CA' \subseteq \CA\big\}.\]
	We define $\rank(X)=\codim(X), \rank(\CA)=\rank(\cap_{H\in \CA}H)$, and \[L_k(\CA) := \{X \in L(\CA)~\vert~\rank(X) = k\}. \] By defining an ordering on $L(\CA)$ by reverse inclusion ($X \leq  Y \Leftrightarrow Y \subseteq X$) the poset $L(\CA)$ becomes a geometric lattice \cite[Lem.~2.3]{MR1217488} which is isomorphic to the lattice of flats of its underlying matroid. 
\end{definition}

We can associate a certain polynomial $\chi(\CA;t)$ to $\CA$, that only depends on its lattice of intersections $L(\CA)$.

\begin{definition}
    Let $\CA$ be an arrangement with intersection poset $L(\CA)$ and Möbius function $\mu$ \cite[Def.~2.42]{MR1217488}. Define the \emph{characteristic polynomial of $\CA$} as $$\chi(\CA;t)=\sum_{X\in L(\CA)}\mu(X)t^{\dim X}.$$ 
\end{definition}

The following operations on $L(\CA)$ are essential.

\begin{definition}
    Let $\CA$ be a hyperplane arrangement. For $X\in L(\CA)$ define the
    \begin{enumerate}
        \item \emph{localization} $\CA_X$ of $\CA$ at $X$ by \[\CA_X = \{H \in \CA~\vert~X \subseteq H\},\text{ and}\]
        \item \emph{restriction} $\CA^X$ of $\CA$ to $X$ by \[\CA^X = \{X \cap H~\vert~H \in \CA\backslash\CA_X~\text{and}~X\cap H \not= \emptyset\}\]  where $\CA^X=(\CA^X, X)$ is an arrangement in $X$.
    \end{enumerate}
\end{definition}
 
 The module of logarithmic vector fields for a simple arrangement $\CA$ was originally introduced by Saito in \cite{MR0586450}. The following definition is its generalization to multiarrangements introduced by Ziegler in \cite{MR1000610}.

\begin{definition}
\noindent 1. A $\mathbb{K}$-linear map $\theta: S \rightarrow S$ is a \emph{derivation} if for $f,g \in S$: \[\theta(f\cdot g) = f\cdot \theta(g)+g\cdot \theta(f)\] Let $\emph{Der}_{\mathbb{K}}(S)$ be the $S$-module of derivations of $S$ graded by polynomial degree \cite[Def.~4.1]{MR1217488}. \\ 2. Define the $S$-submodule $D(\CA,m)$ of $\text{Der}_{\mathbb{K}}(S)$, called the \emph{module of logarithmic vector fields along $(\CA,m)$} (or the module of $(\CA,m)$-derivations), by \[D(\CA,m) = \{\theta \in \text{Der}_{\mathbb{K}}(S)~\vert ~ \theta(\alpha_H)\in \alpha_H^{m(H)}\cdot S~\text{for all}~H\in \CA\}.\]
	The multiarrangement $(\CA,m)$ is called \emph{free} if $D(\CA,m)$ is a free $S$-module.\\
3. For $X\in L(\CA)$ define the \emph{localization} $(\CA_X,m_X)$ of $(\CA,m)$ at $X$, where $m_X=m\vert_{\CA_X}$.
\end{definition}

One of the most interesting questions about the freeness of an arrangement $\CA$ is to what extent the freeness of $\CA$ and the combinatorial structure of $\CA$ are related.
 
 \begin{remark}
 1. If $D(\CA,m)$ is a free $S$-module, then there always exists a homogeneous basis $\{\theta_i\mid 1\leq i\leq \ell\}$ of $D(\CA,m)$ and we call the multiset $\exp(\CA,m)=\{\deg(\theta_i)\mid 1\leq i\leq \ell\}$ the exponents of $(\CA,m)$. It is shown by Ziegler in \cite[Cor.~7]{MR1000610} that if $\rank(\CA)=2$, then $(\CA,m)$ is always free.\\
 2. If $\CA$ is a simple free $\ell$-arrangement with $\CA\neq\emptyset$, then $1\in\exp(\CA)$ \cite[Prop.~4.27]{MR1217488} and if $\exp(\CA)=(1,d_2,\dots,d_\ell)$, then $\chi(\CA;t)=(t-1)(t-d_2)\cdots(t-d_\ell)$, as shown by Terao in \cite{MR0608532}. So $\exp(\CA)$ is determined by $L(\CA)$. If $m\not\equiv 1$, then $\exp(\CA,m)$ is not determined by $L(\CA)$  \cite[Prop.~10]{MR1000610}.
 \end{remark}

Starting at a simple arrangement $\CA$ and fixing one hyperplane $H\in\CA$, Ziegler defined the following multiplicity $m^H$ on $\CA^H$, which only depends on the intersection lattice $L(\CA)$.

\begin{definition}
Let $\CA$ be an arrangement and $H\in\CA$. Define a multiplicity $m^H$ on $\CA^H$ by $$m^H(X)=\vert\CA_X\vert-1,\quad X\in\CA^H$$ and call the pair $(\CA^H,m^H)$ the \emph{Ziegler restriction} of $\CA$ to $H$.
\end{definition} 

Ziegler showed the following connection between the freeness of $\CA$ and the freeness of any of its Ziegler restrictions.

\begin{theorem}[{\cite[Thm.~11]{MR1000610}}]\label{theorem: ziegler restriction}
	Suppose $\CA$ is free with exponents $(1,d_2,\dots,d_\ell)$. Let $H$ be a hyperplane in $\CA$, then $(\CA^H,m^H)$ is free with exponents $(d_2,\dots,d_\ell)$. 
\end{theorem}

Yoshinaga started in \cite{MR3025868} to investigate the question of \emph{extendability}, that is, which free multiarrangements $(\CA,m)$ can be obtained from a free simple arrangement $\CE$ by using Theorem \ref{theorem: ziegler restriction}.

\begin{definition}
Let $(\CA,m)$ be a free $\ell$-multiarrangement. We call a simple $(\ell+1)$-arrangement $\CE$ an \emph{extension of $(\CA,m)$} if there exists a hyperplane $H\in \CE$ such that $(\CE^H,m^H)=(\CA,m)$. If $\CE$ is a free arrangement, then we call it a \emph{free extension of $(\CA,m)$}.
\end{definition} 

In \cite[Ex.~1.4]{MR3025868}, Yoshinaga provides an example of a free rank two multiarrangement $(\CA,m)$ that does not admit a free extension. Note that in Yoshinaga's example, we have $\vert m\vert= 9$. Consequently, Theorem \ref{theorem:MainResult} constitutes the first result establishing the existence of infinitely many free multiplicities on a fixed free arrangement $\CA$ that do not admit any free extensions.\\
Furthermore, Yoshinaga showed in \cite{MR2105827} that if $(\CA^H, m^H)$ is a $2$-arrangement and the multiset of exponents $\exp(\CA^H, m^H)$ is known, then the freeness of $\CA$ - and hence the question of whether $\CA$ is a free extension of $(\CA^H, m^H)$ - depends solely on $L(\CA)$.
  
\begin{theorem}[{\cite[Cor.~3.3]{MR2105827}}]\label{theorem: Yoshinaga ChaPol}
    Let $\CA$ be a $3$-arrangement. Set $$\chi(\CA;t)=(t-1)(t^2-b_1t+b_2)=t^3-(b_1+1) t^2 + (b_1+b_2) t  - b_2$$ and $\exp(\CA^{H},m^H)=(d_1,d_2)$. Then
    \begin{enumerate}
        \item $b_2\geq d_1 d_2$.
        \item If $b_2=d_1d_2$, then $\CA$ is free with exponents $(1,d_1,d_2)$.
    \end{enumerate}
\end{theorem}

It should be noted that our version of Theorem \ref{theorem: Yoshinaga ChaPol} differs slightly from the original statement in \cite{MR2105827} and can be found in (\cite[Thm.~1.39]{MR3205600}). Theorem \ref{theorem: Yoshinaga ChaPol} plays a fundamental role in the construction of free extensions of multiarrangements. 

\begin{remark}\label{Remark: Altering the coefficients of chi(A,t)}  
With notation as in Theorem \ref{theorem: Yoshinaga ChaPol} let $\CA$ be a central rank three arrangement and $\chi(\CA;t)=(t-1)(t^2-b_1t+b_2)=t^3-(b_1+1) t^2 + (b_1+b_2) t  - b_2$. Let $X\in L_2(\CA)$, then $\mu(X)=\vert\CA_X\vert-1$. By definition of $\chi(\CA;t)$, we have  $$b_1+b_2=\sum_{X\in L_2(\CA)}\mu(X)=\sum_{X\in L_2(\CA)}(\vert\CA_X\vert-1).$$ 
By definition of $\chi(\CA;t)$, the leading coefficient of $\chi(\CA;t)$ is always equal to $1$, the coefficient of $t^2$ equals $-\vert\CA\vert$ and the constant term is given by the negative sum of the remaining coefficients of $\chi(\CA;t)$. Consequently, the only coefficient that can get altered is the one in front of $t$, which can be achieved by moving hyperplanes and thereby changing $L(\CA)$ in the process. If $\CA$ is free, then by Theorem \ref{theorem: Yoshinaga ChaPol}, the constant term of $\chi(\CA;t)$ is maximized. In order to achieve this maximum, the coefficient of $t$ must be minimized. By definition of $\chi(\CA;t)$ this coefficient is given by $\sum_{(X\in L_2(\CA))}(\vert\CA_X\vert-1)$.\\

\noindent Define $\mathcal{P}_i:=\vert\{X\in L_2(\CA)\mid\vert\CA_X\vert=i\}\vert$, then we have $$\mathcal{P}_2=\binom{\vert\CA\vert}{2}-\sum_{i>2}\binom{i}{2}\cdot\mathcal{P}_i.$$
Also $$\sum_{X\in L_2(\CA)}(\vert\CA_X\vert-1)=\sum_{i\geq 2} (i-1)\cdot \mathcal{P}_i.$$
Suppose we move a hyperplane in such a way that $\mathcal{P}_3$ increases by one, then $\mathcal{P}_2$ decreases by $\binom{3}{2}=3$. Since in this process, three distinct rank two flats, each containing two elements, are merged into a single rank two flat containing three elements. Therefore, instead of contributing 1 three times (once for each flat with two hyperplanes), the new configuration contributes 2 once (for the flat with three hyperplanes), shrinking the coefficient of $t$ by $1$. Consequently, if $(\CA(A_2),m)$ is the Coxeter multiarrangement of type $A_2$, then constructing a free extension $\CA$ of $(\CA(A_2),m)$ is equivalent to maximizing $\mathcal{P}_3$ for $\CA$.
\end{remark}

Given an arbitrary multiarrangement $(\CA,m)$, Yoshinaga also defined a canonical extension $\CE(\CA,m)$ of $(\CA,m)$; however, this extension is not free in general.

\begin{definition}
Let $(\CA,m)$ be a $\ell$-arrangement. Define the \emph{Yoshinaga-Extension} in $\KK^{\ell+1}$
		with $(x_{1},x_{2},\dots,x_\ell,z)$ as a coordinate system by
		$$\CE(\CA,m)=\{\ker(z)\}\cup\left\{\ker(\alpha_H-kz)\ \middle| \  k\in\mathbb{Z},-\frac{m(H)-1}{2}\leq k\leq \frac{m(H)}{2}\right\}.$$
		Note that for $H=\ker(z)$ we have $(\CE(\CA,m)^H,m^H)=(\CA,m)$.
\end{definition} 

\section{Exponents for the Coxeter multiarrangement of type $B_2$}
For a general multiarrangement $(\CA, m)$, determining wheter it is free and, if so, identifying its exponents is, in general, a difficult problem. Even in the case where $\CA$ is a Coxeter arrangement of rank two, a complete classification of the exponents $\exp(\CA, m)$ for arbitrary multiplicities $m$ is currently available exclusively for the Coxeter arrangement of type $A_2$, as established by Wakamiko in \cite{MR2328057}. In this section, we present the latest results concerning the classification of $\exp(\CA(B_2), m)$, where $\CA(B_2)$ denotes the Coxeter arrangement of type $B_2$ equipped with an arbitrary multiplicity $m$. Understanding the exponents of $(\CA(B_2),m)$ is essential for applying Theorem \ref{theorem: Yoshinaga ChaPol}.

\begin{definition}
Fixing a multiplicity $m=(m_1,m_2,m_3,m_4)$, the \emph{Coxeter multiarrangement $\CA(B_2)$ of type $B_2$} is defined by 
$$Q(\CA(B_2),m)=x^{m_1}y^{m_2}(x-y)^{m_3}(x+y)^{m_4}.$$ 
We will denote the underlying simple arrangement by $\CA(B_2)$ for the rest of this article.
\end{definition} 

\begin{definition}
The multiplicity $m=(m_1,m_2,\ldots,m_n)$ on an arrangement $\CA$ of $n$ hyperplanes is called \textit{balanced} if $m$ satisfies $m_i<\frac{\vert m\vert}{2}$ for all $1\leq i\leq n$, where $\vert m\vert:=\sum_{i=1}^{n}m_i$. Otherwise call $m$ \textit{unbalanced}.
\end{definition} 

\begin{remark}\label{Remark: Unbalanced m exponents}
Wakefield and Yuzvinsky \cite{MR2309190} determined the exponents of multiarrangements $(\CA,m)$ with unbalanced multiplicities in the case $\rank(\CA)=2$. They showed that if $m_{\max}=\max\{m(H)\mid H\in \CA\}$ satisfies $m_{\max}\geq\frac{\vert m\vert}{2}$, then the exponents are given by $\exp(\CA,m)=(\vert m\vert -m_{\max}, m_{\max})$. 
Consequently, in order to complete the classification of $\exp(\CA(B_2),m)$, it remains to consider the case where $m$ is a balanced multiplicity. 
\end{remark}

The following theorem is instrumental in constraining the possible values of $\exp(\CA, m)$ when analyzing the exponents of a central $2$-arrangement $\CA$ with a balanced multiplicity $m$.

\begin{theorem}[{\cite[Thm.~1.6]{MR3070120}}]\label{theorem: peak points}
    Let $\CA$ be a central $2$-arrangement over a field of characteristic zero with $\vert\CA\vert>2$. If $m:\CA\to\ZZ_{>0}$ is balanced and $\exp(\CA,m)=(d_1,d_2)$, then $\Delta(m):=\vert d_1-d_2\vert\leq \vert\CA\vert-2$.
\end{theorem}

\begin{remark}\label{Remark: peak points and A2 classification}
    \begin{enumerate}
        \item With notation as in Theorem \ref{theorem: peak points} the multiplicity $m$ is called a \emph{peak point} of $\CA$ if $\vert d_1-d_2\vert=\vert\CA\vert-2$. Additional information on the combinatorial importance of peak points can be found in \cite{MR2873095}.
        \item Recall that the Coxeter multiarrangement $\CA(A_2)$ of type $A_2$ is defined by $$Q(\CA(A_2),m)=x^{m_1}y^{m_2}(x-y)^{m_3}.$$ Therefore $\vert \CA\vert=3$ and if $m$ is balanced, then Theorem \ref{theorem: peak points} shows that $$\exp(\CA(A_2),m)=(\frac{\vert m\vert}{2},\frac{\vert m\vert}{2})\text{ if }\vert m\vert\text{ is even or }\exp(\CA(A_2),m)=(\lfloor\frac{\vert m\vert}{2}\rfloor,\lceil\frac{\vert m\vert}{2}\rceil)\text{ if }\vert m\vert\text{ is odd.}$$
        By combining this with Remark \ref{Remark: Unbalanced m exponents} in the case where $m$ is unbalanced, we recover the full classification of $\exp(\CA(A_2),m)$ for arbitrary multiplicities $m$, originally established by Wakamiko in \cite{MR2328057}.
    \end{enumerate}
\end{remark}

Let $m$ be a multiplicity on $\CA(B_2)$. Theorem \ref{theorem: peak points} enables us to constrain the possible values of $\exp(\CA(B_2),m)$ as follows.

\begin{corollary}\label{corollary: deltas for m}
        Let $m$ be a balanced multiplicity on $\CA(B_2)$. Then we have
        \begin{enumerate}
            \item $|m|$ is odd if and only if $\Delta(\CA(B_2),m)=1$.
            \item $|m|$ is even if and only if $\Delta(\CA(B_2),m)\in\{0,2\}$. 
        \end{enumerate}
\end{corollary}

It is sufficient to classify $\exp(\CA(B_2),m)$ up to some permutations.

\begin{remark}\label{remark:permutation}
    For a multiplicity $m=(m_1,m_2,m_3,m_4)$ on $\CA=\CA(B_2)$, 
    one can show that the multiarrangement $(\CA,m)$ is isomorphic to the multiarrangement $(\CA,m_{perm})$, where $$m_{perm}\in\{(m_1,m_2,m_4,m_3), (m_2,m_1,m_3,m_4), (m_3,m_4,m_1,m_2)\}.$$ 
    In particular, this implies that $\exp(\CA,m)=\exp(\CA,m_{perm})$. Therefore, when classifying the exponents of $(\CA(B_2),m)$ for an arbitrary multiplicity $m$, it is sufficient to assume, without loss of generality, that $m_2-m_1 \geq m_4-m_3 \geq 0$.
\end{remark}

The following corollary can be established using results by Maehara and Numata \cite{2312.06356}, as well as Feigin, Wang and Yoshinaga \cite{2309.01287}. Since its proof is independent of the main development of this article, we defer it to the \hyperref[appendix]{Appendix}, where the relevant results are also reviewed.

\begin{corollary}\label{corollary: deletion exponents appendixproof}
    Let $(\CA,m)=(\CA(B_2),(m_1,m_2,m_3,m_4))$ with $m$ balanced. Let $\min\{m_i\mid1\leq i\leq4\}=m_4=1$. Then $m$ is a peak point if and only if $m_1=m_2$, $m_3\in 2\ZZ+1$ and $|m|\in 4\ZZ$. 
\end{corollary}

We illustrate the implications of this corollary in the following example.

\begin{example}\label{example: k equal to 3 is peakpoint}
Consider the multiarrangement $(\CA(B_2),m_k)=(\CA(B_2),(2,k,1,k))$ and let $k=3$. Fix an arbitrary hyperplane $H\in\CA(B_2)\backslash\{\ker(x-y)\}$. Applying Corollary \ref{corollary: deletion exponents appendixproof}, we obtain $\exp(\CA(B_2),(2,2,1,3))=(3,5)$, whereas $\exp(\CA(B_2),(1,3,1,3))=\exp(\CA(B_2),(2,3,1,2))=(4,4)$.
Now consider the case $k\geq 4$. In this situation, reducing the multiplicity of any hyperplane $H\in\CA(B_2)\backslash\{\ker(x-y)\}$ by $1$ will not yield a peak point for $\CA(B_2)$, since the condition $m_1=m_2$ from Corollary \ref{corollary: deletion exponents appendixproof} is not satisfied.
\end{example}

\section{Reformulating Yoshinaga's Theorem}
In this section, we present the central tool utilized in the proof of Theorem \ref{theorem:MainResult} and illustrate its application through several examples. For a multiarrangement $(\CA,m)$, the concepts of locally and globally mixed products were introduced by Abe, Terao, and Wakefield in \cite{ATW07}. We adapt these definitions to suit the context of our work as follows.

\begin{definition}
1. For any $X\in L_2(\CA)$, the localization $(\CA_X,m_X)$ is free with exponents $\exp(\CA_X,m_X)=(d_1^X,d_2^X)$. The \emph{locally mixed product} of $(\CA,m)$ is then defined by
		\[\LMP (\CA,m)=\sum_{X\in L_2(\CA)} d_1^Xd_2^X.\]   
2. If $(\CA,m)$ is free with exponents $\exp(\CA,m)=(d_1,\dots,d_\ell)$, the \emph{globally mixed product} is defined by \[\GMP (\CA,m)=\sum d_{i_1}d_{i_2},\] where the sum is taken over all unordered pairs $\{d_{i_1},d_{i_2}\}\subset\exp(\CA,m)$.\\
If $m \equiv 1$ (i.e. $\CA$ is a simple arrangement), we abbreviate by writing $\LMP(\CA)$ and $\GMP(\CA)$ instead of $\LMP(\CA,m)$ and $\GMP(\CA,m)$, respectively.
\end{definition} 

Since $\GMP(\CA,m)$ is defined only for free arrangements, we introduce the following generalization.

\begin{definition} 
Let $\CA$ be a simple arrangement of rank three and let $H\in\CA$. Then the Ziegler restriction $(\CA^H,m^H)$ is free. If $\exp(\CA^H,m^H)=(d_1,d_2)$, we define the \emph{virtually globally mixed product} of $\CA$ with respect to $H$ by \[\VGMP(\CA, H)=d_1+d_2+d_1d_2.\]
Note that if $\CA$ is free, then by Theorem \ref{theorem: ziegler restriction} and Theorem \ref{theorem: Yoshinaga ChaPol} we have $\GMP(\CA)=\VGMP(\CA,H)$ for all $H\in\CA$.
\end{definition} 

Using this terminology, we may reformulate Theorem \ref{theorem: Yoshinaga ChaPol} as follows. 

\begin{lemma}\label{Lemma: LMP>=GMP for extensions}
Let $\CA$ be a central $3$-arrangement, and let $H\in\CA$. Then $$\LMP(\CA)\geq \VGMP(\CA, H).$$
In particular, $\CA$ is free if and only if $\LMP(\CA)=\VGMP(\CA, H)$ for all $H\in\CA$.
\end{lemma}

\begin{proof}
    Let $H\in\CA$ and suppose $\exp(\CA^H,m^H)=(d_1,d_2)$. Since the localization $\CA_X$ is free for all $X\in L_2(\CA)$, the locally mixed product $\LMP(\CA)$ is always defined. As $\CA$ is a simple arrangement, it holds that if $\vert \CA_X\vert=d$, then $\exp(\CA_X)=(1,d-1)$. Therefore, $$\LMP(\CA)=\sum_{X\in L_2(\CA)}(\vert\CA_X\vert-1)=\sum_{X\in L_2(\CA)}\mu(X).$$ This sum corresponds to the coefficient of $t$ in $\chi(\CA;t)$. We now adapt the notation of $\chi(\CA;t)$ as in Theorem \ref{theorem: Yoshinaga ChaPol}, and compare coefficients. From the definition of $\chi(\CA;t)$ and the fact that $\CA$ is a simple arrangement with $\rank(\CA)=3$, the coefficient of $t^2$ is $b_1=\vert\CA\vert-1$. On the other hand Theorem \ref{theorem: Yoshinaga ChaPol} implies that $$\vert\CA\vert=1+d_1+d_2 \implies b_1=d_1+d_2.$$ Moreover, since $b_1+b_2=\LMP(\CA)$, and Theorem \ref{theorem: Yoshinaga ChaPol} gives the inequality $b_2\geq d_1d_2$, we obtain $$\LMP(\CA)=b_1+b_2\geq d_1+d_2+d_1d_2=\VGMP(\CA,H).$$ Equality holds if and only if $b_2=d_1d_2$, which, by Theorem~\ref{theorem: Yoshinaga ChaPol}, is equivalent to the freeness of $\CA$. 
    \end{proof}

The proof of Lemma \ref{Lemma: LMP>=GMP for extensions} gives rise to the following Corollary.

\begin{corollary}
    Let $\CA$ be a central $3$-arrangement. Then its characteristic polynomial is given by $$\chi(\CA;t)=t^3-\vert\CA\vert t^2+\LMP(\CA) t-b_2.$$
\end{corollary}

We now recall the definition of the \emph{deletion} $(\CA',m')$ of a multiarrangement $(\CA,m)$.
\begin{definition}
     Let $(\CA,m)$ and let $H\in\CA$. The \emph{deletion of $H$} is defined as the pair $(\CA',m')$, where:
\begin{center}
\begin{itemize}
\item if $m(H) > 1$, define $\CA':=\CA$ and $m'(H) := m(H)-1, m'(H') := m(H')$ for all $H'\in\CA\backslash\{H\}$.\\ 
\item if $m(H) = 1$, define $\CA':=\CA\backslash\{H\}$ and $m'(H') := m(H')$ for all $H'\in\CA\backslash\{H\}$.
\end{itemize}
\end{center}
\end{definition}

The following observations simplify the proof of Corollary \ref{corollary: upper bound for restriction hyperplanes}.

\begin{lemma}\label{Lemma: deletion unbalanced then no peak point}
    Let $(\CA,m)$ be a multiarrangement with balanced multiplicity $m$, and suppose that $\exp(\CA,m)=(e_1,e_2)$. Assume there exists a hyperplane $H_0\in\CA$ such that the deletion $(\CA',m')$ of $H_0$ yields an unbalanced multiplicity $m'$. Then the following statements hold:
    \begin{enumerate}
        \item $\vert m\vert$ is odd.
        \item If $(\CA,m)=(\CA(B_2),m)$, then $\exp(\CA',m')=(\frac{\vert m\vert-1}{2},\frac{\vert m\vert-1}{2})$.
    \end{enumerate}
\end{lemma}
\begin{proof}
    First let $(\CA,m)$ with $m=(m_1,m_2,\dots,m_k)$ and without loss of generality assume $m_1=1+\sum_{2\leq i\leq k}m_i$. Then $\vert m\vert=\sum_{1\leq i\leq k}m_i=1+2\cdot(\sum_{2\leq i\leq k}m_i)$ which is an odd number.\\
    Now let $(\CA,m)=(\CA(B_2),m)$ such that $m$ is a balanced multiplicity and there exists a hyperplane $H_0\in\CA$ such that for the deletion $(\CA',m')$ of $H_0$ the multiplicity $m'$ is unbalanced. Let $m'_{\max}:=\max\{m'(H)\mid H\in \CA'\}$, then thanks to Remark \ref{Remark: Unbalanced m exponents} $m'_{\max}\in\exp(\CA',m')$.\\
    Applying Lemma \ref{Lemma: deletion unbalanced then no peak point}(1) shows that $\vert m\vert$ is odd and since $m$ is balanced use Corollary \ref{corollary: deltas for m} to derive $\exp(\CA,m)=(\frac{\vert m\vert-1}{2},\frac{\vert m\vert+1}{2})$. Since $m'_{\max}\in\exp(\CA',m')$ we require $\exp(\CA',m')=(\frac{\vert m\vert-1}{2},\frac{\vert m\vert-1}{2})$, otherwise we get a contradiction to the assumptions on $m$ or $m'$ as follows. If $\exp(\CA',m')=(\frac{\vert m\vert-3}{2},\frac{\vert m\vert+1}{2})$ and $m'_{\max}=\frac{\vert m\vert-3}{2}<\frac{\vert m\vert-1}{2}=\frac{\vert m'\vert}{2}$, then $m'$ is not unbalanced by definition. If $m'_{\max}=\frac{\vert m\vert+1}{2}$, then $m$ was already unbalanced.
\end{proof}

Given a free multiarrangement $(\CA,m)$ with $\exp(\CA,m)=(d_1,d_2)$, we extend the notation $\Delta(m):=\vert d_1-d_2\vert$, originally introduced in Theorem \ref{theorem: peak points} for balanced multiplicities, to also include the case of non balanced multiplicities. We derive the following tool from Lemma \ref{Lemma: LMP>=GMP for extensions} and Lemma \ref{Lemma: deletion unbalanced then no peak point}. 

\begin{corollary}\label{corollary: upper bound for restriction hyperplanes}
    Let $\CE$ be a free extension of $(\CA,m)=(\CA(B_2),m)$ for a balanced multiplicity $m$ such that $(\CE^{H_0},m^{H_0})=(\CA,m)$. Fix $H\in\CE\setminus{H_0}$ and let $\exp(\CA',m')$ be the deletion of $(H\cap H_0)$ from $(\CA,m)$ and $\CE'=\CE\backslash\{H\}$. Then we have
    \begin{enumerate}
        \item If $\vert m\vert$ is odd, then
            \begin{enumerate}
                \item $\vert \CE^{H}\vert\leq \frac{r+3}{2}$ (if $\Delta(m')=2$) or
                \item $\vert \CE^{H}\vert\leq \frac{r+1}{2}$ (if $\Delta(m')=0$).
            \end{enumerate}
        \item If $\vert m \vert$ is even, then 
        \begin{enumerate}
            \item $\vert\CE^{H}\vert\leq \frac{r}{2}$ (if $\Delta(m)=2$) or 
            \item $\vert \CE^{H}\vert\leq \frac{r+2}{2}$ (if $\Delta(m)=0$).
        \end{enumerate} 
    \end{enumerate}
    In particular, the deletion $\CE'$ of $H$ is free if and only if equality holds.
\end{corollary}
\begin{proof}
    Assume first that $m'$ is unbalanced. Then, by Lemma \ref{Lemma: deletion unbalanced then no peak point}, $\vert m\vert$ is odd and $\Delta(m')=0$. Thus, this case falls under $(1)(b)$ of Corollary \ref{corollary: upper bound for restriction hyperplanes}. All statements concerning exponents are based on Remark \ref{Remark: Unbalanced m exponents} and Corollary \ref{corollary: deltas for m}.\\ Let $(\CE^{H_0},m^{H_0})=(\CA,m)$ and suppose first that $r:=\vert m\vert$ is odd. Since $\CE$ is free, we have $$\exp(\CE)=\left(1,\frac{r-1}{2}, \frac{r+1}{2}\right),\text{ and by Lemma \ref{Lemma: LMP>=GMP for extensions}, }\LMP(\CE)=\frac{r^2-1}{4} + r.$$ Now fix $H\in\CE\backslash\{H_0\}$, and note that $\exp(\CA',m')\in\{(\frac{r-1}{2}, \frac{r-1}{2}), (\frac{r-3}{2}, \frac{r+1}{2})\},$ depending on whether $\Delta(m')=0$ or $\Delta(m')=2$. By Lemma \ref{Lemma: LMP>=GMP for extensions}, we obtain: $$\LMP(\CE')\geq\VGMP(\CE',H_0)\in\left\{\frac{r^2-3}{4} + \frac{r}{2}, \frac{r^2-7}{4} + \frac{r}{2} \right\}.$$ Since $\CE$ is central, we have $$\LMP(\CE)=\sum_{X\in L_2(\CE)}(\vert\CE_X\vert-1)$$ from which it follows that $$\LMP(\CE')=\LMP(\CE)-\vert \CE^H\vert.$$ Substituting into the previous inequality yields: $$\frac{r^2-1}{4} + r - \vert \CE^H\vert=\LMP(\CE')\geq \VGMP(\CE',H_0)= \frac{r^2-3}{4} + \frac{r}{2},$$ or respectively, 
    $$\frac{r^2-1}{4} + r - \vert \CE^H\vert=\LMP(\CE')\geq \VGMP(\CE',H_0)= \frac{r^2-7}{4} + \frac{r}{2}.$$
    This implies: $$\vert\CE^H\vert\leq\frac{r+1}{2}\text{ or (if } \Delta(m')=2)\mbox{ } \vert\CE^H\vert\leq\frac{r+3}{2}.$$ with equality being equivalent to freeness for $\CE'$, again thanks to Lemma \ref{Lemma: LMP>=GMP for extensions}.\\
    Now suppose that $r$ is even. Then $$\exp(\CA(B_2),m)\in\left\{(\frac{r}{2},\frac{r}{2}),(\frac{r}{2}-1,\frac{r}{2}+1) \right\},$$ and by Lemma \ref{Lemma: LMP>=GMP for extensions}, it follows that $$\LMP(\CE)\in\left\{\frac{r^2 }{4} + r, \frac{r^2 }{4} + r - 1\right\}.$$ Now delete $(H_0\cap H)$ from $(\CA(B_2),m)$. Since the resulting multiplicity $m'$ remains balanced, we obtain $$\exp(\CA(B_2)',m')=\left(\frac{r-2}{2},\frac{r}{2}\right)\text{ and }\VGMP(\CE',H_0)=\frac{r^2}{4}+\frac{r}{2}-1.$$ We now consider the two possible cases for $\LMP(\CE)$: 
    $$\frac{r^2 }{4} + r - \vert \CE^H\vert=\LMP(\CE')\geq \VGMP(\CE',H_0)=\frac{r^2}{4} + \frac{r}{2} - 1$$
    or $$\frac{r^2 }{4} + r- 1 - \vert \CE^H\vert=\LMP(\CE')\geq \VGMP(\CE',H_0)=\frac{r^2}{4} + \frac{r}{2} - 1.$$
    In the first case (i.e. when $\Delta(m)=0$), this implies $\vert\CE^H\vert\leq \frac{r+2}{2},$ and in the second case (i.e. when $\Delta(m)=2$), $\vert\CE^H\vert\leq \frac{r}{2}.$ Moreover, by Lemma \ref{Lemma: LMP>=GMP for extensions}, equality holds if and only if $\CE'$ is free.
\end{proof}

Yoshinaga proved in \cite{MR2077250} and \cite{MR2105827} a famous conjecture of Edelman and Reiner: for every Coxeter multiarrangement $(\CA(W),m)$, where $m$ is a constant multiplicity, the corresponding Yoshinaga extension $\CE(\CA(W),m)$ is free. However, for arbitrary multiplicities $m$, the Yoshinaga-Extension of $(\CA(B_2),m)$ is not necessarily free. Nevertheless, there may exist various other free extensions, as illustrated by the following example.

\begin{example}\label{example: free non-yoshinaga extensions}
    Let $(\CA(B_2),m)$ with $m=(3,5,2,2)$, then $\exp(\CA(B_2),m)=(5,7)$. Suppose there exists a free extension $\CE$ with a hyperplane $H\in \CE$ such that $(\CE^H,m^H)=(\CA(B_2),m)$. In this case, it holds that $\VGMP(\CE,H)=47$. Figure \ref{fig:Yoshinaga-Extension of (3,5,2,2)} shows the deconing of the Yoshinaga-Extension $\CE(\CA(B_2),m)$, for which $\LMP(\CE(\CA(B_2),m))=49$. Thus, by Lemma \ref{Lemma: LMP>=GMP for extensions}, it is not free. In Figure \ref{fig:Free Extensions of (3,5,2,2)}, we present two alternative extensions of $(\CA(B_2),m)$, both of which are free and have non-isomorphic intersection lattices. 
\end{example}

Example \ref{example: free non-yoshinaga extensions} leads to the following conclusion.

\begin{corollary}
    Let $\CE_1,\CE_2$ be free extensions of the multiarrangement $(\CA(B_2),m)$. Then $L(\CE_1)$ and $L(\CE_2)$ need not be isomorphic.
\end{corollary}

\begin{figure}
       \centering
       \includegraphics[width=0.5\linewidth]{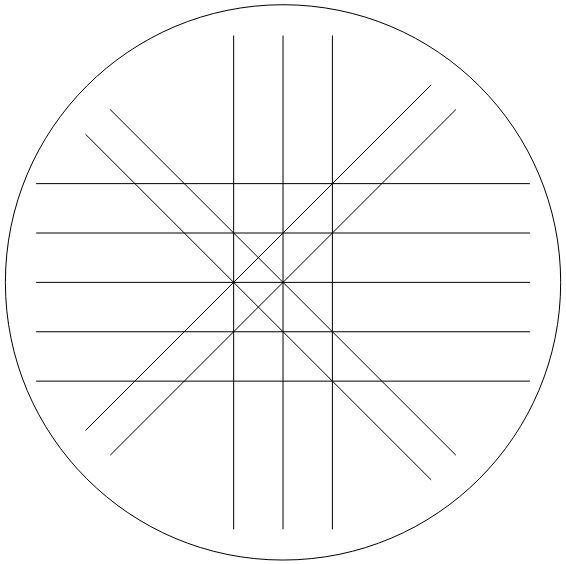}
       \caption{The Yoshinaga-Extension $\CE(\CA(B_2),(3,5,2,2))$.}
       \label{fig:Yoshinaga-Extension of (3,5,2,2)}
   \end{figure}
   
   \begin{figure}
       \centering
       \includegraphics[width=0.4\linewidth]{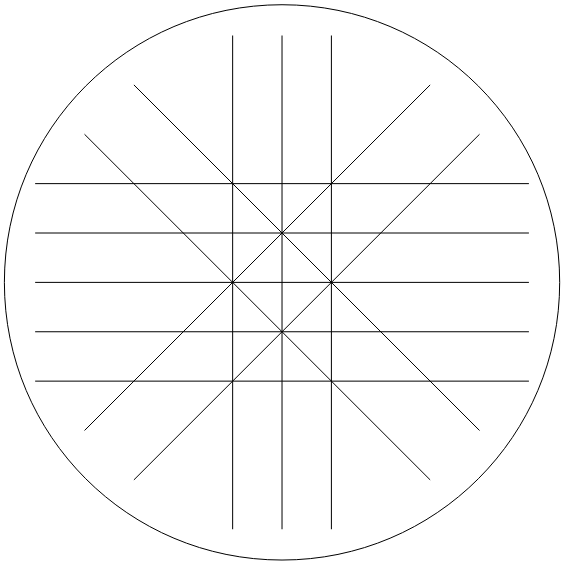}
       \includegraphics[width=0.4\linewidth]{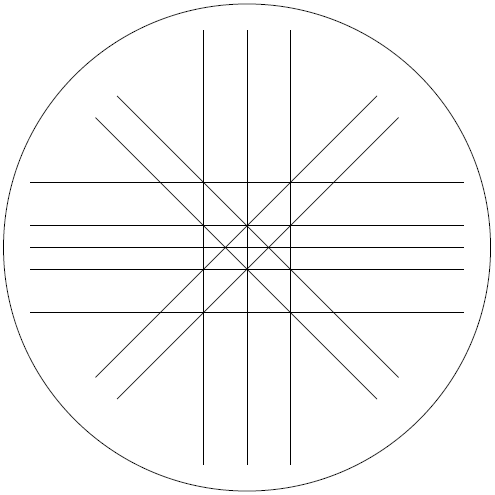}
       \caption{Free extensions of $(\CA(B_2),(3,5,2,2))$ with non-isomorphic lattices of intersection.}
       \label{fig:Free Extensions of (3,5,2,2)}
   \end{figure}

\section{Proof of Theorem \ref{theorem:MainResult}}
We require the following lemma for the proof of Theorem \ref{theorem:MainResult}.

\begin{lemma}\label{lemma: main theorem deletion free}
    Let $k\ge 1$ and $(\CA,m)$ be the multiarrangement over $\BBQ$ defined by $Q(\CA,m) = x^ky^k(x-y)^2$. Then every free extension 
of $(\CA,m)$ is, up to a change of basis, the Yoshinaga extension of $(\CA,m)$.
\end{lemma}
\begin{proof}
Let $\CE$ be a free extension of $(\CA,m)$, and let $K$ be the corresponding field extension of $\BBQ$.
Let $H_0 \in \CE$ such that $(\CE^{H_0},m^{H_0})=(\CA,m)$, and let $\rho \colon \CE\backslash\{H_0\} \rightarrow \CA$ be the associated restriction map. Define $\tilde H_x := \ker x$, $\tilde H_y := \ker y$ and $\tilde H_{x-y} = \ker (x-y)$. There exist $H_1,H_2 \in \CE$ such that $\rho^{-1} (\tilde H_{x-y}) = \{H_1,H_2\}$; we refer to these as the copies of $\tilde H_{x-y}$. Similarly, we define $H_{x,1}, \ldots, H_{x,k} \in \CE$ and $H_{y,1},\ldots, H_{y,k}\in \CE$ as the copies of $\tilde H_x$ and $\tilde H_y$, respectively. 

Note that the subarrangement $\CE_0 = \CE \setminus \{H_1,H_2\}$ is an extension of the Boolean multiarrangement $(\CB,m_0)$, defined by $Q(\CB,m_0) = x^ky^k$. Hence, $\CE_0$ is free with $\exp(\CE_0) = (1,k,k)$. Furthermore, by Remark \ref{Remark: peak points and A2 classification} (2), we have $\exp(\CA,m) = (k+1,k+1)$, and therefore $\exp(\CE) = (1,k+1,k+1)$. 
By \cite[Thm.~1.2]{2306.11310}, there exists a free path between $\CE_0$ and $\CE$. We may assume without loss of generality that $\CE_0 \cup \{H_1\}$ is free (otherwise we relabel $H_1$ and $H_2$).
Thus, $\vert (\CE_0\cup \{H_1\})^{H_1} \vert = k+1$ and $\vert \CE^{H_2} \vert = k+2$.

We call $X \in L_2(\CE)$ a triple point if $X \not\subset H_0$ and $\vert\CE_X\vert = 3$. Based on the restriction sizes above, there are exactly $k$ triple points $P_1$ with $H_1 \in \CE_{P_1}$, and
exactly $k-1$ triple points $P_2$ with $H_2 \in \CE_{P_2}$. The triple points of $H_1$ contain all copies of $\tilde H_x$ and $\tilde H_y$, i.e., 
$$
U_1 := \{H \in \CE_P \mid P \text{ is a triple point of } H_1 \} = \{\text{copies of } \tilde H_x \} \cup \{\text{copies of } \tilde H_y \} \cup \{H_1\}.
$$
For $H_2$, define 
$$
U_2 := \{H \in \CE_P \mid P \text{ is a triple point of } H_2 \}. 
$$
Then exactly one copy of $H_x$ and one copy of $H_y$ are missing from $U_2$. For each triple point $P$
on $H_2$, there exists a unique triple point $P_x$ on $H_1$ such that $(\CE_P \cap \CE_{P_x}) = \{\text{copy of } \tilde H_x\}$, and similarly a unique triple point $P_y$ on $H_1$ such that 
$(\CE_{P} \cap \CE_{P_y}) = \{\text{copy of } \tilde H_y\}$. 

Choose a triple point $P_2$ on $H_2$ and the unique corresponding point $P_1$ on $H_1$. After a suitable base change, we may assume that $H_0 := \ker z$, and that 
$H_{x,i} = x - \lambda_i z$ and $H_{y,i} = y - \tilde \lambda_i z$ for $i=1,\ldots,k$, with suitable $\lambda_i,\tilde \lambda_i \in K$. We now apply the deconing construction (setting $z=1$) to visualize the intersection pattern. Now $P_1,P_2$ can be viewed as points in a plane. After a translation we may assume $P_1 =(0,0)$ and after scaling the $z$-component we furthermore may assume that $P_2 = (1,0)$. Both operations are base changes.

The initial intersection pattern is illustrated below.

\begin{center}
\begin{tikzpicture}
    \draw[dashed] (0,-4) -- (0,4); \node[draw=none] at (0,4.5) {$x=0$};
    \draw[dashed] (2,-4) -- (2,4); \node[draw=none] at (2,4.5) {$x=1$};
    \draw[dashed] (-2,-2) -- (4,4); \node[draw=none] at (-2.5,-2) {$H_1$};
    \draw[dashed] (-2,-4) -- (4,2); \node[draw=none] at (-2.5,-4) {$H_2$};
    \draw[dashed] (-2,0) -- (4,0); \node[draw=none] at (5,0) {$y=0$};
    \draw[dashed] (-2,2) -- (4,2); \node[draw=none] at (5,2) {$y=1$};
    \draw[draw,-triangle 90] (0,0) -- (2,0); \node[draw=none] at (5,0)  {};
    \draw[draw,-triangle 90] (2,0) -- (2,2); \node[draw=none] at (5,0)  {};
    \node[draw=none] at (-0.3,0.3) {$P_1$};
    \node[draw=none] at (1.6,0.3) {$P_2$};
    \node[draw=none,gray] at (1.6,2.3) {$P_3$};
    \node[draw=none,red] at (3.6,2.3) {$P_4$};
    \draw[draw,-triangle 90,red] (2,2) -- (4,2); \node[draw=none] at (5,0)  {};
\end{tikzpicture}
\end{center}

All dashed lines correspond to hyperplanes in $\CE$. Clearly, the next triple point must be $P_3 = (1,1)$. If $k\geq 3$, further triple points are required; the next one on $H_2$ is $P_4 = (2,1)$, implying that the copy $x = 2$ (namely $\ker (x-2z)$) of $\tilde H_x$ is in $\CE$, and so forth.

So in each double step (up, right) we increase the point by $(1,1)$ on $H_2$. Similarly, backtracking to the source (left, down) we decrease the point by $(1,1)$. Therefore, it is not possible to revisit a point in characteristic zero.
\end{proof}

Although it is not possible to revisit a certain point in the proof of Theorem \ref{theorem:MainResult} when working over a field of characteristic zero, this becomes a feasibility in positive characteristic, as illustrated by the following example.

\begin{example}\label{Example: Extension not unique in positive characteristic}
Let $K$ be the finite field with $\vert K \vert = 9$, and let $w \in K \setminus \{0,1,2\}$. Furthermore, let $(\CA,m)$ be the multiarrangement defined by $Q(\CA,m) = x^4y^4(x-y)^2$ over $K$.
Then the arrangement $\CE$ over $K$ defined by $$Q(\CE) = z (x-y)(x-y+z)x(x+z)(x+2z)(x+wz) y(y+z)(y+2z)(y+wz)$$ is a free extension of $(\CA,m)$.
\end{example}

\begin{figure}
\begin{tikzpicture}
 \fill [black] ( -2, -2) circle [radius=0.05]node[above  right]{}; 
 \fill [black] (-1,-3) circle [radius=0.05]node[below  left]{$(0,0)$}; origin
 \fill [black] ( 0,-3) circle [radius=0.05]node[below  left]{}; 
 \fill [black] (-1,-2) circle [radius=0.05]node[below  left]{}; 

  \fill [black] (-3, -1) circle [radius=0.05]node[above  left]{}; 
  \fill [black] (-1, 0) circle [radius=0.05]node[above  right]{$(0,3)$}; 
  \fill [black] (-1,-1) circle [radius=0.05]node[above  left]{}; 
  \fill [black] ( 0,-1) circle [radius=0.05]node[above  right]{$(1,2)$}; 
  \fill [black] ( 0,-2) circle [radius=0.05]node[above  left]{}; 
  
  \fill [black] (-2, 0) circle [radius=0.05]node[above  left]{}; 
  \fill [black] (-2,-1) circle [radius=0.05]node[above  left]{}; 
  \fill [black] (-3, 0) circle [radius=0.05]node[above  left]{}; 

  \fill [black] ( -4,0) circle [radius=0.05]node[above  left]{}; 
  \fill [black] ( 1,-2) circle [radius=0.05]node[above  left]{}; 
  \fill [black] ( 1,-3) circle [radius=0.05]node[above  left]{}; 
  \fill [black] ( 2,-3) circle [radius=0.05]node[above  left]{}; 

  \draw[domain=-4:2] plot(0,\x)node[above right]{$\ker(x-z)$}; 
  \draw[domain=-4:2] plot(-1,\x)node[above]{$\ker(x)$}; 
  
  \draw[domain=-4.5:3] plot(\x,0)node[right]{$\ker(y-3z)$}; 
  \draw[domain=-4.5:3] plot(\x,-1)node[right]{$\ker(y-2z)$}; 
  \draw[domain=-4.5:4] plot(\x,-2)node[right]{}; 
  \draw[domain=-4.5:4] plot(\x,-3)node[right]{}; 
  
  \draw[domain=0:-4.5] plot(\x,-\x-4)node[above]{$\ker(x+y)$}; 
  \draw[domain=-3:3] plot(\x,-\x-1)node[right]{}; 
  \draw[domain=-3.5:2] plot(\x,-\x-2)node[right]{}; 
  \draw[domain=1:-3.5] plot(\x,-\x-3)node[above left]{}; 

  \draw (0,-1) circle[radius=6]node[above]{};
  \draw (0,4.2)node[below ]{$\ker(z)$}; 
\end{tikzpicture}
\caption{Deconing ($z=1$) of free extension of $(\CA(A_2),(2,4,4))$.}\label{fig: free extension of A_2}
\end{figure}
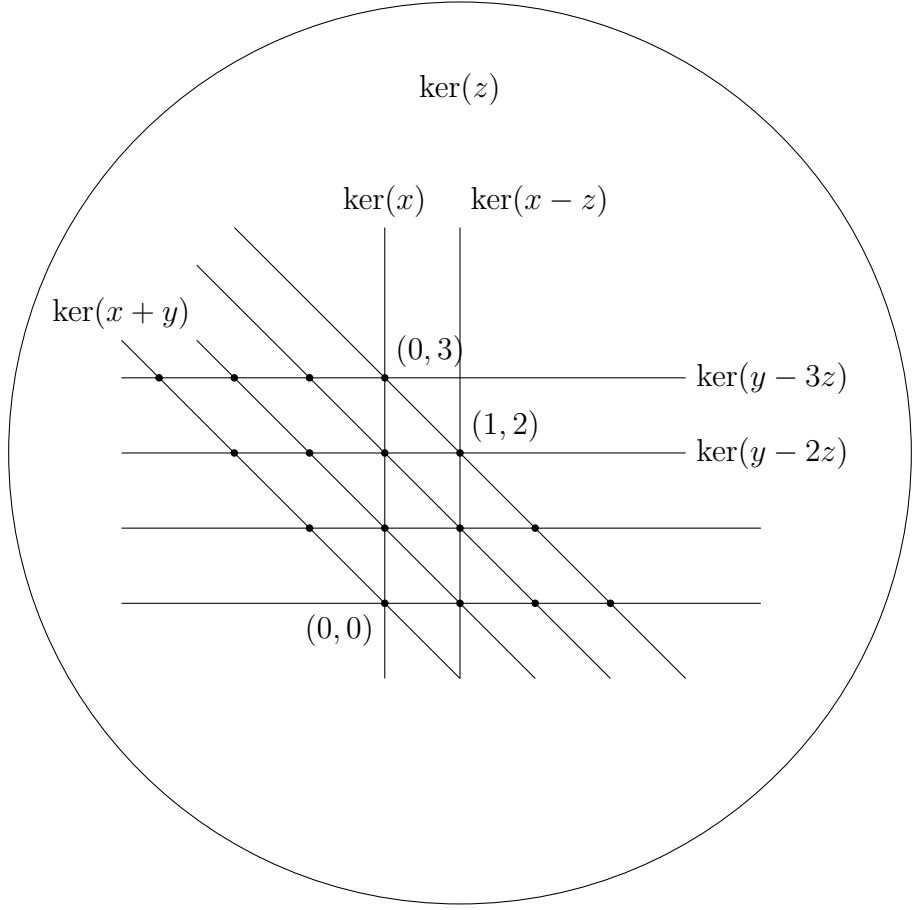

Given an extension $\CE$ of $(\CA(B_2),m)$, the following results provide upper and lower bounds for $\vert \CE^H\vert$ for any hyperplane $H\in\CE\backslash\{H_0\}$.

\begin{lemma}\label{lemma: hyperplane cardinality of extension restriction}
    Let $\ell\geq 2$ and $\CA$ be an $\ell$-arrangement. Let $m$ be a multiplicity on $\CA$, and let $\CE$ be an extension of $(\CA,m)$; that is, there exists a hyperplane $H_0\in\CE$ such that the Ziegler restriction to $H_0$ is $(\CA,m)$. Let $ H_1\in\CE\backslash\{H_0\}$ be an arbitrary hyperplane, and define $m_1:=\max\{m(H\cap H_0)\mid H\in\CE\backslash\{H_0,H_1\}\}$, then $\vert\CE^{H_1}\vert\geq 1+m_1$. 
\end{lemma}

\begin{proof}
 Let $H_0, H_1, \CE$ be as in Lemma \ref{lemma: hyperplane cardinality of extension restriction}, and let $H\in\CA$ such that $m(H)=m_1$. Define the set $C_H:=\{H'\in\CE \mid (H'\cap H_0) = H\}\subset\CE,$ which consists of $m_1$ hyperplanes. Since $\CE^{H_0}=\CA$, it follows that for any $H'\in C_H$ that $\CE_{(H_0 \cap H_1)} \cap \CE_{(H_0 \cap H')}=\emptyset$, as otherwise the hyperplanes would get identified in $\CA$. Furthermore, for any $H'\in C_H$, we have $\CE_{(H_0\cap H')}=C_H\cup\{H_0\}$. In particular, this implies that $\CE_{(H_1\cap H')}\cap C_H=\emptyset$. Therefore, it holds that $(\{H_1\cap H'\mid H'\in C_H\}\cup\{H_0\cap H_1\})\subset \CE^{H_1}$, which completes the proof.
\end{proof}

Combining Corollary \ref{corollary: upper bound for restriction hyperplanes} and Lemma \ref{lemma: hyperplane cardinality of extension restriction}, we obtain upper and lower bounds for $\vert\CE^H\vert$, where $\CE$ is a free extension $\CE$.

\begin{corollary}\label{corollary: restriction on amount of hyperplanes for extension of B2}
    Let $\CE$ be a free extension of $(\CA,m)=(\CA(B_2),m)$, where $m$ is a balanced multiplicity with $r:=\vert m\vert$, and suppose $(\CE^{\ker(z)},m^{\ker(z)})=(\CA,m)$. Let $H_1\in\CE\backslash\{\ker(z)\}$ be arbitrary, $\exp(\CA,m)=(e_1,e_2), \exp(\CA',m')=(d_1,d_2)$, and $m_1:=\max\{m(H\cap \ker(z))\mid H\in\CE\backslash\{\ker(z),H_1\}\}$. Then:
    \begin{enumerate}
        \item If $r$ is odd, then: 
        \begin{enumerate}
            \item $1+m_1\leq \vert \CE^{H_1}\vert\leq \frac{r+3}{2}$ (if $\vert d_1-d_2\vert=2$).
            \item $1+m_1\leq \vert \CE^{H_1}\vert\leq \frac{r+1}{2}$.
        \end{enumerate} 
        \item If $r$ is even, then: 
        \begin{enumerate}
            \item $1+m_1\leq \vert\CE^{H_1}\vert\leq \frac{r}{2}$ (if $\vert e_1-e_2\vert=2$).
            \item $1+m_1 \leq \vert\CE^{H_1}\vert\leq \frac{r+2}{2}$.
        \end{enumerate} 
    \end{enumerate}
\end{corollary}

\begin{remark}\label{remark: using corollary for nonfreeness}
Corollary \ref{corollary: restriction on amount of hyperplanes for extension of B2} can be employed to demonstrate that certain multiarrangements $(\CA(B_2),m)$ do not admit a free extension $\CE$, by constraining the possible values of $\vert\CE^H\vert$ in attempts to disprove the existence of such an extension. With the same notation as in Corollary \ref{corollary: restriction on amount of hyperplanes for extension of B2}, consider the multiarrangement $$(\CA,m)=(\CA(B_2),(2,k,1,k)), (k\geq 4),$$ and assume $\CE$ is a free extension of $(\CA,m)$. Using Corollary \ref{corollary: deletion exponents appendixproof}, we observe that removing any hyperplane $H\in\CA\backslash\{\ker(z)\}$ does not create a peak point for $\CA(B_2)$, while Corollary \ref{corollary: restriction on amount of hyperplanes for extension of B2} (1) gives $k+1\leq\vert\CE^H\vert\leq k+2$. From Corollary \ref{corollary: deltas for m}, we find $\exp(B_2,m)=(k+1,k+2)$, hence by Theorem \ref{theorem: Yoshinaga ChaPol}, $\exp(\CE)=(1,k+1,k+2)$. Since $\rank(\CE^H)=2$, we have $\exp(\CE^H)=(1,\vert\CE^H\vert-1)$. Applying the Addition-Deletion-Theorem \cite{Add-Del-Terao}, we are left with two possibilities.
\begin{enumerate}
        \item $\vert\CE^H\vert= k+2\iff \CE'$ is free.
        \item $\vert\CE^H\vert= k+1\iff \CE'$ is not free.
    \end{enumerate}
    A case-by-case analysis (based on whether specific deletions $\CE'$ of $\CE$ are free) shows that no such extension $\CE$ exists which satisfies the inequality $k+1\leq\vert\CE^H\vert\leq k+2$ for every hyperplane $H\in\CE\backslash\{\ker(z)\}$. Note that the assumption $k\geq 4$ is crucial here. If we instead assume $k=3$, then one of the deletions is a peak point, as shown in Example \ref{example: k equal to 3 is peakpoint}. In that case, the alternative part of Corollary \ref{corollary: restriction on amount of hyperplanes for extension of B2} (1) must be applied.
\end{remark}

We are now ready to prove our main result.
\begin{proof}[Proof of Theorem \ref{theorem:MainResult}]
    We continue to use the notation introduced in the proof of Lemma \ref{lemma: main theorem deletion free}. Let $\CE$ be a free extension of $(\CA(B_2),(2,k,1,k))$ with $k\geq 4$. First, suppose that the deletion $\CE':=\CE\backslash\{H\}$, where $H$ denotes the unique copy of $\tilde H_{x-y}$, is free. Then $\CE'$ is a free extension of $(\CA(A_2),(2,k,k))$, and by Lemma \ref{lemma: main theorem deletion free}, we may assume that $\CE'=\CE(\CA(A_2),(2,k,k))$ is the Yoshinaga-Extension. As in Lemma \ref{lemma: main theorem deletion free}, we employ the deconing construction (i.e. set $z=1$) to visualize the corresponding arrrangement. After a suitable change of basis, we may assume that the intersection points of the first copy $H_1\in\CE'$ of $\tilde H_x$ with the copies of $\tilde H_y$ are given by $I_1=\{(0,i)\mid 0\leq i\leq k-1\},$ and those of the second copy $H_2\in\CE'$ by $I_2=\{(1,i)\mid 0\leq i\leq k-1\}$. Since $\CE'$ is the Yoshinaga-Extension, $k-1$ copies of $\tilde H_{x+y}$ each contain two triple points of $\CE'$, and the remaining one contains a single triple point. These $k-1$ copies each contain a pair of points $(0,i)\in I_1$ and $(1,i-1)\in I_2$, while the final hyperplane contains either $(0,0)\in I_1$ or $(1,k-1)\in I_2$. The corresponding configuration for $k=4$ is illustrated in Figure \ref{fig: free extension of A_2}.\\
    With our chosen coordinates, the copies of $\tilde H_y$ are given by the hyperplanes $\ker(y-a_iz)$ for $0\leq a_i\leq k-1$, and the copies of $\tilde H_{x+y}$ are given by $\ker(x+y-b_jz)$ for $0\leq b_j\leq k-1$. If the intersection point of $H$ with $\ker(y)$ is $(x_1,y_1)$, then the intersection point of $H$ with $\ker(y-a_iz)$ is $(x_1+a_i,y_1+a_i)$. On the other hand, if the intersection point of $H$ with $\ker(x+y)$ is $(x_2,y_2)$, then the intersection point of $H$ with $\ker(x+y-b_jz)$ is $(x_2+\frac{b_j}{2},y_2+\frac{b_j}{2})$.\\
    This shows that at least half of the intersection points of $H$ with copies of $\tilde H_y$ cannot lie on any copy of $\tilde H_{x+y}$. Since $\CE$ also includes the intersection of $H$ with $\ker(z)$, and $k\geq 4$, we obtain $$\vert\CE^H\vert\geq k+1+\left\lfloor\frac{k}{2}\right\rfloor\geq k+3,$$ which contradicts Remark \ref{remark: using corollary for nonfreeness}, where it is shown that we must have $\vert\CE^H\vert=k+2$.
    
    \noindent Now suppose instead that $\CE'$ is not free. Then $\vert\CE^H\vert=k+1$, which means that $\CE^H$ contains the intersection point $(H\cap\ker(z))$ as well as the $k$ intersections of $H$ with the copies of $\tilde H_y$. Hence, each intersection point on $H$ is a triple intersection point formed with one copy each of $\tilde H_y$ and $\tilde H_{x+y}$ from the Boolean extension. 

From the fact that $\vert\CE^H\vert=k+1$, it follows that $\vert \CE^{H'} \vert = k+2$ for $H' \in \{H_1,H_2\}$, where $H_1$ and $H_2$ denote the two copies of $\tilde H_x$. This implies that each of these hyperplanes has exactly $k-1$ triple and $2$ double intersection points with the Boolean extension. One of these double points arises from an intersection with a single copy of $\tilde H_y$, and the other one from an intersection with a single copy of $\tilde H_{x+y}$.
 
Since our focus in on the intersection points with the Boolean extension, we may again work in the deconing (i.e., set $z=1$). After a change of basis, we may assume that the special intersection points of $H$ with $H_1$ and $H_2$ are $(0,0)$ and $(1,1)$ in the deconing (see Figure \ref{fig1}). In particular, the copies of $\tilde H_x$ in $\CE$ are given by $\ker(x)$ and $\ker(x-z)$.

\begin{figure}
\centering
\begin{minipage}[t]{.4\textwidth}
\begin{tikzpicture}[scale=1.5]
    \draw[] (-1.5,0) -- (2.5,0);  
    \draw[] (-1.5,1) -- (2.5,1);  

    \draw[] (0,-1.5) -- (0,2.5);  \node[draw=none] at (0,2.65) {$H_1$} ; 
    \draw[] (1,-1.5) -- (1,2.5);  \node[draw=none] at (1,2.65) {$H_2$} ; 
    \node[draw=none] at (-0.1,-0.3) {$0$};

    \draw[] (-1.5,1.5) -- (1.5,-1.5);
    \draw[] (-0.5,2.5) -- (2.5,-0.5);

    \draw[draw] (-1.5,-1.5) -- (2.5,2.5); \node[draw=none] at (2.5,2.65) {$H$} ; 
\end{tikzpicture}\caption{}\label{fig1}
\end{minipage}
\begin{minipage}[t]{.4\textwidth}
\begin{tikzpicture}[scale=1.5]
    \draw[] (-1.5,0) -- (2.5,0);  
    \draw[] (-1.5,1) -- (2.5,1);  

    \draw[] (0,-1.5) -- (0,2.5);  \node[draw=none] at (0,2.65) {$H_1$} ; 
    \draw[] (1,-1.5) -- (1,2.5);  \node[draw=none] at (1,2.65) {$H_2$} ; 
    \node[draw=none] at (-0.1,-0.3) {$0$};

    \draw[] (-1.5,1.5) -- (1.5,-1.5);
    \draw[] (-0.5,2.5) -- (2.5,-0.5);

    \draw[draw] (-1.5,-1.5) -- (2.5,2.5); \node[draw=none] at (2.5,2.65) {$H$} ; 
    
    \draw[draw, -triangle 90, green,line width=1.2pt] (0,1) -- (1,1);
    \draw[draw, -triangle 90, green,line width=1.2pt] (1,1) -- (0,2);
    \draw[draw, -triangle 90, green] (0,0.5) -- (0.5,0.5);
    \draw[draw, -triangle 90, green] (0.5,0.5) -- (0,1);
    \draw[draw, -triangle 90, green] (0,0.25) -- (0.25,0.25);
    \draw[draw, -triangle 90, green] (0.25,0.25) -- (0,0.5);

    \draw[draw,dashed] (-1.5,0.25) -- (0,0.25);  
    \draw[draw,dashed] (0.25,0.25) -- (2.5,0.25);  

    \draw[draw,dashed] (-1.5,0.5) -- (0,0.5);  
    \draw[draw,dashed] (0.5,0.5) -- (2.5,0.5);      

    \draw[draw,dashed] (-1.5,2) -- (0,0.5);  
    \draw[draw,dashed] (0,0.5) -- (2,-1.5);      

    \draw[draw,dashed] (-1.5,2.5) -- (0,1);  
    \draw[draw,dashed] (0.5,0.5) -- (2.5,-1.5);

    \draw[draw, -triangle 90, red,line width=1.2pt] (1,0) -- (0,0);
    \draw[draw, -triangle 90, red,line width=1.2pt] (0,0) -- (1,-1);
\end{tikzpicture}\caption{}\label{fig2}
\end{minipage}
\end{figure}

Observe that the intersection points $(0,1)$ and $(0,2)$ on $H_1$ are connected via the intersection point $(1,1)$ on $H$, where we consider a path that traverses portions of $H$ and a copy of $\tilde H_x$, rather than remaining entirely on $H_1$.
Similarly, the intersection points $(1,0)$ and $(1,-1)$ on $H_2$ are connected via the intersection point $(0,0)$ on $H$.
This observation motivates the definition of the following graphs $G(H')$ for $H' \in \{H_1, H_2\}$.

The vertices of $G(H')$ are the intersection points on $H'$ and the intersection points on $H$. Recall that the elements of $\CE^H$ are the intersections of $H$ with the $k$ copies of $\tilde H_y$ (each of which also lies in a copy of $\tilde H_{x+y}$), along with the intersection point of $H$ with $\ker(z)$. Since $\CA$ is a central arrangement, for each intersection point $P$ of $H$ with a copy $H_{C_y}$ of $\tilde H_y$, there exists a corresponding intersection point $P_{-}$ of $H'$ with $H_{C_y}$, such that $P$ and $P_{-}$ are connected via $H_{C_y}$.
Moreover, since $P$ lies on a copy $H_{C_{x+y}}$ of $\tilde H_{x+y}$, there also exists an intersection point $P_{+}$ on $H'$ such that $P$ and $P_{+}$ are connected via $H_{C_{x+y}}$.
For every such triple $(P_{-},P,P_+)$ we add directed edges $P_{-} \rightarrow P$ and $P \rightarrow P_+$ in the graph $G(H')$.
It is important to note that all sources and sinks of this graph must lie on $H'$, since every intersection point on $H$ lies on in both a copy of $\tilde H_y$ and a copy of $\tilde H_{x+y}$.
Consequently, a sink in $G(H')$ corresponds to an intersection point on $H'$ not intersected by any copy of $\tilde H_y$, while a source corresponds to an intersection point on $H'$ not intersected by any copy of $\tilde H_{x+y}$.

We begin by analyzing the structure of $G(H_1)$. Let $p = (0,c)$ be an intersection point on $H_1$, where $c$ may lie in an extension field of $\BBQ$. The intersection points on $H_1$ which can lie in the same connected components of $G(H_1)$ as $p$ are given by $p_k = (0,2^kc)$ for $k \in \BBZ$. Thus, the graph $G(H_1)$ contains only one single cycle, namely the loop at $(0,0)$. All other connected component of $G(H_1)$ must contain both a source and a sink. Recall that there are exactly two double intersection points with the Boolean arrangement on $H_1$, while every other intersection point on $H_1$ is a triple intersection point. Therefore, one of the double points must serve as the sink and the other one as the source of the single remaining connected component of $G(H_1)$. In particular, this component must contain the points $(0,1)$ and $(0,2)$. 

The argument involving $G(H_2)$ is similar. Let $p = (1,c)$ be an intersection point on $H_2$, where $c$ might be an element of an extension field of $\BBQ$. Then the intersection points on $H_2$ that can lie in the same connected component as $p$ are $p_k = (1,2^kc-(2^k-1))$ for $k \in \BBZ$. Consequently, $G(H_2)$ also contains one single cycle, which is the loop at $(1,1)$, and thus it, too, has two connected components. The second component must contain the points $(1,0)$ and $(1,-1)$.

Next, observe that $(0,2)$ must be a sink in $G(H_1)$; otherwise $(1,2)$ would introduce another connected component on $G(H_2)$.
From this, it follows that the structure of $G(H_1)$ is fully determined, as all remaining edges must be predecessors of $(0,0)$ on $G(H_1)$. We illustrate the information contained in the next sentences in Figure \ref{fig2}. Now, consider the point $(0,\frac{1}{2})$, which must be a triple point on
$H_1$, as there must be at least three triple points on $H_1$. This implies that $(0,\frac{1}{4})$ is another intersection point on $H_1$, intersected by $\ker(y-\frac{1}{4})$. However, this in turn implies that $(1,\frac{1}{4})$ is a vertex on $H_2$, and thus a vertex in the graph $G(H_2)$. But this intersection point of $H_2$ would span another connected component on $G(H_2)$, which leads to a contradiction.\\
Finally, we note that for $k=3$, the construction described in the second half of the proof yields a free extension. 
\end{proof}
    
If $m$ is an unbalanced multiplicity on $\CA(B_2)$, then the exponents $\exp(\CA(B_2),m)$ are known. However, existing results on the classification of $\exp(\CA(B_2),m)$ are largely limited to partial results obtained by Feigin, Wang, and Yoshinaga \cite{2309.01287}, as well as Maehara and Numata \cite{2312.06356}. Even without a complete description of the exponents, it may still be possible to generalize the proof strategy of Theorem \ref{theorem:MainResult} to a broader class of multiplicities, since the range of possible exponent combinations for $(\CA(B_2),m)$ is quite restricted.  
\begin{question}
    Is it possible to fully classify when a free extension exists of $(\CA(B_2),m)$ with an arbitrary multiplicity $m$?
\end{question}

\section{Implications for higher dimensions}
Theorem \ref{theorem:MainResult} has direct implications for the existence of free extensions of higher-dimensional arrangements, due to the following result by Abe, Nuida, and Numata. 

\begin{theorem}[{\cite[Prop.~1.7]{abenuidanumata:signedeliminable}}]\label{theorem: Multilocalizations are free}
	Let $(\CA,m)$ be a multiarrangement. For $X\in L(\CA)$ the localization $(\CA_X,m_X)$ at $X$ is free provided that $(\CA,m)$ is free. 
\end{theorem}

By combining Theorem \ref{theorem:MainResult} and Theorem \ref{theorem: Multilocalizations are free}, we obtain the following result. 

\begin{corollary}\label{corollary: implications for higher Bn}
    Let $(\CA,m)=(\CA(B_n),m)$ be the Coxeter multiarrangement of type $B_n$. If there exists a $X\in L(\CA)$, with $\rank(X)=2$ and $(\CA_X,m_X)=(\CA(B_2),m_k)$ (where $k\geq 4$ and $(\CA(B_2),m_k)$ as in Theorem \ref{theorem:MainResult}), then no free extension of $(\CA,m)$ exists.
\end{corollary}

\begin{proof}
    Let $(\CA,m)=(\CA(B_n),m)$ and suppose $\CE$ is an extension of $(\CA,m)$. Fix $X \in L(\CA)$ such that $(\CA_X,m_X)=(\CA(B_2),(2,k,1,k))$. Then $\CE_X$ is an extension of $(\CA_X,m_X)=(\CA(B_2),(2,k,1,k))$. By Theorem \ref{theorem:MainResult} $(\CA_X,m_X)$ does not admit a free extension. Finally, Theorem \ref{theorem: Multilocalizations are free} implies that $\CE_X$ cannot free.
\end{proof}

To construct an example illustrating Corollary \ref{corollary: implications for higher Bn} in the case $n=3$, we make use of the following theorem by Abe, Terao and Wakefield.

\begin{theorem}[{\cite[Thm.~5.10]{MR2400395}}]\label{theorem: free vertex}
Suppose $\CA$ is supersolvable with a filtration $\CA=\CA_r\supset\CA_{r-1}\supset\dots\supset\CA_2\supset\CA_1$ and $r\geq 2$.
Let $m$ be a multiplicity on $\CA$ and let $m_i = m\vert_{\CA_i}$. Let $\exp(\CA_2,m_2)=(d_1,d_2,0,\dots,0)$. Suppose that for each $H'\in \CA_d\backslash \CA_{d-1}$, $H''\in\CA_{d-1}$ $(d=3,\dots,r)$ and $X:=(H'\cap H'')$, we either have $$\CA_X=\{H',H''\}$$ or 
$$m(H'')\geq\left(\sum_{X\subset H\in \CA_d\backslash\CA_{d-1}}m(H)\right)-1.$$ Then $(\CA,m)$ is inductively free with $\exp(\CA,m)=(d_1,d_2,\vert m_3\vert-\vert m_2\vert,\dots,\vert m_r\vert-\vert m_{r-1}\vert,0,\dots,0)$.
\end{theorem}

We apply Corollary \ref{corollary: implications for higher Bn} together with Theorem \ref{theorem: free vertex} to obtain the following family of multiplicities on the Coxeter arrangement $\CA(B_3)$ of type $B_3$, for which no free extensions exist.

\begin{corollary}\label{corollary: B3 non extendable example}
    Let $$Q(\CA(B_3),n_k)=x^2 y^k (x-y)^1 (x+y)^k z^e (x-z)^f (x+z)^g (y-z)^h (y+z)^i, (k\geq 4),$$ be the Coxeter multiarrangement of type $B_3$. Fix the supersolvable filtration $$\{\ker(x)\}\subset \{\ker(x), \ker(y), \ker(x-y), \ker(x+y)\}\subset \CA(B_3).$$ Then $(\CA(B_3),n_k)$ satisfies the conditions of Theorem \ref{theorem: free vertex} if and only if $e=f=g=h=i=1$. Moreover, each $(\CA(B_3),n_k)$ is inductively free with exponents $\exp(\CA(B_3),n_k)=(5,k+1,k+2)$. However, no free extension of $(\CA(B_3),n_k)$ exists.
\end{corollary}
\begin{proof}
    We fix the supersolvable filtration $\{\ker(x)\}\subset \{\ker(x), \ker(y), \ker(x-y), \ker(x+y)\}\subset \CA(B_3)$ The relevant localizations for Theorem \ref{theorem: free vertex} yield the following inequalities:
    \begin{center}
    $\{\ker(x),\ker(z),\ker(x-z),\ker(x+z)\}$ which gives $a\geq e+f+g-1$.\\
    $\{\ker(x-y),\ker(x-z),\ker(y-z)\}$ which gives $c\geq f+h-1$.\\
    $\{\ker(x-y),\ker(x+z),\ker(y+z)\}$ which gives $c\geq g+i-1$.
    \end{center}
    Now let $(a,b,c,d)=(2,k,1,k)$, where $k\geq 4$. Substituting into the inequalities we obtain $2\geq e+f+g-1$, $1\geq f+h-1$, and $1\geq g+i-1$. The second and third inequalities imply $f=g=h=i=1$. Substituting into the first inequality yields $2\geq e+1+1-1=e+1$, so $e=1$. Therefore, the conditions of Theorem~\ref{theorem: free vertex} are satisfied if and only if $e = f = g = h = i = 1$. Finally, by Corollary~\ref{corollary: implications for higher Bn}, we conclude that no free extension of $(\mathcal{A}(B_3), n_k)$ exists.
\end{proof}

\begin{remark}
    It is well known that for every Coxeter multiarrangement $(\CA(A_2),m)$ of type $A_2$, a free extension always exists. In \cite{MR3025868}, Yoshinaga provided a canonical free extension of a multiarrangement $(\CA,m)$ of rank greater than two, under the condition that $\CA$ is locally $A_2$, along with other constraints on both $\CA$ and $m$. This result raises the natural expectation that similar canonical free extensions might be constructed for other rank two arrangements.\\
    However, our main result indicates that this is not generally the case: even when moving from type $A_2$ to type $B_2$, the construction of free extensions may fail in many instances. As demonstrated in Corollary \ref{corollary: implications for higher Bn}, this obstruction extends to arrangements of higher rank, where the existence of a single non-extendable localization can prevent the existence of any free extension.
\end{remark}

\section*{Acknowledgements}{The authors would like to express their sincere gratitude to Takuro Abe and Gerhard Röhrle for directing them to relevant literature important to this project and for their insightful feedback throughout its development. Furthermore, the authors are grateful for their careful reading of the manuscript and for their valuable suggestions, which have significantly improved the quality of this work.}

\section{Appendix}\label{appendix}
In this appendix, we recall two fundamental results before presenting the proof of Corollary \ref{corollary: deletion exponents appendixproof}. The first of these results, due to Maehara and Numata, is stated below.
\begin{corollary}[{\cite[Thm.~3.3,~Cor.~3.4]{2312.06356}}]\label{thm:main}
Let $(\CA(B_2),m)$ be a multiarrangement where $m=(m_1,m_2,m_3,m_4)$ is balanced. 
If $m$ satisfies either $|m_4-m_3|\leq 1$ or $|m_4-m_3|=2$ and $m_1,m_2\in2\ZZ+1$, then 
$m$ is a peak point if and only if $m_4-m_3=0$, $m_1$,$m_2\in 2\ZZ+1$ and $|m|\in 4\ZZ$. 
\end{corollary}

We now recall a result of Feigin, Wang and Yoshinaga from \cite{2309.01287}. Throughout we restrict our attention to balanced multiplicities $m=(m_1,m_2,m_3,m_4)$ on $\CA(B_2)$ with $m_4=\min\{m_i\mid 1\leq i \leq 4\}\leq2$.
\begin{theorem}[{\cite[Thm.~2.3]{2309.01287}}]
    \label{theorem:FWY}
    For $a,b,c\in\ZZ_{>0}$, let us define $\theta_{a,b,c}\in\Der(S)$ to be 
    $$\theta_{a,b,c}:=\left(\int^{x} t^c(t-x)^b(t-y)^a dt\right) \partial_{x} 
    + \left(\int^{y} t^c(t-x)^b(t-y)^a dt\right) \partial_{y}.$$ 
    Using this derivation, 
    
    \begin{itemize}
        \item 
    for any multiarrangement $(\CA(A_2),(p,q,r))$ defined by $Q(\CA(A_2),\mu):={x}^p{y}^q(x-y)^r$ with $p+q+r$ being odd and its multiplicity being balanced, 
        $D(\CA(A_2),(p,q,r))$ is generated by the two derivations $\theta_{a,b,c}$ and $\theta_{a,b,c+1}$ where 
        $$(a,b,c):=(\frac{-p+q+r-1}{2},\frac{p-q+r-1}{2},\frac{p+q-r-1}{2}).$$ 

        \item
    for any multiarrangement $(\CA(A_2),(p,q,r))$ defined by $Q(\CA(A_2),\mu):={x}^p{y}^q(x-y)^r$ with $p+q+r$ being even and its multiplicity being balanced, 
        $D(\CA(A_2),(p,q,r))$ is generated by the two derivations $x\cdot\theta_{a,b,c}$ and $y\cdot\theta_{a-1,b+1,c}$ where 
        $$(a,b,c):=(\frac{-p+q+r}{2},\frac{p-q+r-2}{2},\frac{p+q-r-2}{2}).$$ 
    \end{itemize}
Furthermore, for each of such triples $(a,b,c)\in\ZZ^3$, define $I(a,b,c)\in\QQ$ to be 
$$I(a,b,c):=(\theta_{a,b,c}(x+y))_{x=1,y=-1}=\int^{1}_{0} t^c(t-1)^b(t+1)^a dt+\int^{-1}_{0} t^c(t-1)^b(t+1)^a dt.$$
\end{theorem}

We begin by establishing the following for use in the proof.

\begin{proposition}\label{proposition:differentialformula}
Under the definitions  $\theta_{a,b,c}$ and $I(a,b,c)$, 
\begin{enumerate}
    \item $\theta_{a,b,c}(x+y)\in(x+y)S$ if and only if $I(a,b,c)=0$.
    \item $I(a,b,c)=0$ if and only if $a=b$ and $a+b+c\in2\ZZ$.
\end{enumerate}
\end{proposition}

\begin{proof}
$(1)$ Let $\theta_{a,b,c}(x+y) \in (x+y)S$. Substituting $x = 1$ and $y = -1$, it follows that $I(a,b,c) = 0$. Conversely, assume that $I(a,b,c) = 0$. Since $p(x,y) := \theta_{a,b,c}(x + y)$ is a homogeneous polynomial of some degree $n$, it can be written in the form
    $$
    p(x, y) = \theta_{a,b,c}(x + y) = \sum_{i=0}^{n} \lambda_i (x + y)^i x^{n-i},
    $$
    for some constants $\lambda_0, \ldots, \lambda_n$. Evaluating $p$ at $(1, -1)$ yields
    $$
    p(1, -1) = \lambda_0 \cdot 1 = 0,
    $$
    which implies $\lambda_0 = 0$, and hence $\theta_{a,b,c}(x + y) \in (x + y)S$.

    $(2)$ Assume without loss of generality that $a \leq b$, and set $n = b - a$. The case $b < a$ can be treated analogously. We compute:
    \begin{align*}
    I(a,b,c)
    &= \int_{0}^{1} t^c (t - 1)^b (t + 1)^a \, dt + \int_{0}^{-1} t^c (t - 1)^b (t + 1)^a \, dt \\
    &= \int_{0}^{1} 
    \underbrace{t^c (t - 1)^a (t + 1)^a}_{=: f(t)}
    \underbrace{\left((t - 1)^n - (-1)^{a + b + c} (t + 1)^n \right)}_{=: g(t)} \, dt.
    \end{align*}
    The function $f(t)$ does not change sign on the interval $(0,1)$. The same is true for $g(t)$, since for $t \in (0,1)$ and $n > 0$, we have $(t + 1)^n > (t - 1)^n$. Therefore, $I(a,b,c) = 0$ if and only if $g(t) = 0$ for all $t \in (0,1)$. This is only the case when $n = 0$ and $a + b + c \in 2\mathbb{Z}$. In particular, $n = 0$ implies $a = b$.
\end{proof}

\begin{proof}[Proof of Corollary \ref{corollary: deletion exponents appendixproof}]
    Assuming that $m_4=1$, we have 
    $Q(\CA(B_2),m-\delta_{H_4})={x}^{m_1}{y}^{m_2}(x-y)^{m_3}=Q(\CA(A_2),(m_1,m_2,m_3))$. By Theorem \ref{theorem:FWY}, the multiplicity $m$ is a peak point if and only if 
        $\theta_{a,b,c}(x+y)\in(x+y)S,$ where $(a,b,c):=(\frac{-m_1+m_2+m_3-1}{2},\frac{m_1-m_2+m_3-1}{2},\frac{m_1+m_2-m_3-1}{2})$.
        By Proposition \ref{proposition:differentialformula} (1) and (2), this  condition holds if and only if 
        $a=b$ and $a+b+c\in2\ZZ$, which is equivalent to 
        $m_1=m_2$ and $m_1+m_2+m_3+1\in4\ZZ$. Since this also implies that $m_3\in2\ZZ+1$, these are precisely equivalent to those stated in Corollary \ref{thm:main}, up to some permutation as noted in Remark \ref{remark:permutation}.  
\end{proof}

\bibliographystyle{amsalpha}
\newcommand{\etalchar}[1]{$^{#1}$}
\providecommand{\bysame}{\leavevmode\hbox to3em{\hrulefill}\thinspace}
\providecommand{\MR}{\relax\ifhmode\unskip\space\fi MR }

\end{document}